\documentclass{amsart}
\usepackage{amscd,amssymb,amsxtra}

\theoremstyle{plain}
\newtheorem{Thm}{Theorem}[section]

\newtheorem{Lem}[Thm]{Lemma}
\newtheorem{Cor}[Thm]{Corollary}
\newtheorem{Claim}[Thm]{Claim}
\newtheorem*{PET}{The Pointwise Ergodic Theorem}

\theoremstyle{definition}
\newtheorem{Def}[Thm]{Definition}

\newtheorem{Remark}[Thm]{Remark}
\newtheorem{Ex}[Thm]{Example}

\numberwithin{equation}{section}

\newcommand{\Alt}{\operatorname{Alt}}

\newcommand{\Fix}{\operatorname{Fix}}

\newcommand{\sgn}{\operatorname{sgn}}

\newcommand{\supp}{\operatorname{supp}}
\newcommand{\Irr}{\operatorname{Irr}}
\newcommand{\mn}{\text{min}}

\newcommand{\Sub}{\operatorname{Sub}}
\newcommand{\Sym}{\operatorname{Sym}}
\newcommand{\Fin}{\operatorname{Fin}}
\newcommand{\mx}{\text{max}}
\newcommand{\ER}{\operatorname{ER}}
\newcommand{\can}{\text{can}}
\newcommand{\Diag}{\operatorname{Diag}}

\begin{document}


\begin{abstract}
We will describe the relationship 
between the indecomposable characters of $\Fin(\mathbb{N})$ 
and its ergodic invariant random subgroups; and we
will interpret each Thoma character
$\chi_{(\beta; \gamma)}$
as an asymptotic limit of a naturally associated sequence 
of characters induced from linear characters of Young
subgroups of finite symmetric groups.
\end{abstract}

\title[Characters and invariant random subgroups]
{Characters and invariant random subgroups of the 
finitary symmetric group} 

\author{Simon Thomas}
\address
{Mathematics Department \\
Rutgers University \\
110 Frelinghuysen Road \\
Piscataway \\
New Jersey 08854-8019 \\
USA}
\email{simon.rhys.thomas@gmail.com}


\maketitle

\section{Introduction} \label{S:intro}

If $\Omega$ is an infinite set, then the corresponding
{\em finitary symmetric group\/} $\Fin(\Omega)$ is
the group of permutations $g \in \Sym(\Omega)$
such that 
$\supp(g) = 
\{\, \omega \in \Omega \mid g(\omega) \neq \omega \,\}$
is finite. In his classic paper \cite{t}, Thoma
classified the indecomposable characters of 
$\Fin(\mathbb{N})$. More recently, Vershik \cite{v2}
classified the ergodic invariant random subgroups
of $\Fin(\mathbb{N})$; and he 
pointed out that the indecomposable characters of $\Fin(\mathbb{N})$ are very closely connected with its ergodic invariant random subgroups. In this paper, we will 
describe the precise relationship between the indecomposable characters of $\Fin(\mathbb{N})$ and its ergodic invariant random subgroups. 
Before we stating our main result, we will recall
Thoma's classification of the the indecomposable characters 
of $\Fin(\mathbb{N})$ and Vershik's 
classification\footnote{We will take this opportunity to correct an inaccuracy in the statement \cite{v2} of Vershik's classification theorem.}  of the ergodic invariant subgroups of $\Fin(\mathbb{N})$.
Throughout this paper, $D[\,0,1\,]$ will denote the set of sequences
\[
\alpha = (\, \alpha_{1}, \alpha_{2}, \cdots, 
\alpha_{n}, \cdots\,) \in [\,0,1\,]^{\mathbb{N}^{+}}
\] 
such that $\alpha_{1} \geq \alpha_{2} \geq \cdots \geq 
\alpha_{n} \geq \cdots$ 

First recall that if $G$ is a countable group, then a function $\chi: G \to \mathbb{C}$ is said to be a
{\em character\/} if the following conditions are satisfied:
\begin{enumerate}
\item[(i)] $\chi(h\,g\,h^{-1}) = \chi(g)$ for all $g$, $ \in G$.
\item[(ii)] $\sum_{i,j=1}^{n} \lambda_{i} \Bar{\lambda}_{j} \chi(g_{j}^{-1}g_{i}) \geq 0$
for all $\lambda_{1}, \cdots, \lambda_{n} \in \mathbb{C}$ and $g_{1}, \cdots, g_{n} \in G$.
\item[(iii)] $\chi(1_{G}) = 1$.
\end{enumerate}
A character $\chi$ is said to be {\em indecomposable\/} or 
{\em extremal\/} if it is impossible to express $\chi = r \chi_{1} + (1 - r) \chi_{2}$, where 
$0 < r < 1$ and $\chi_{1} \neq \chi_{2}$ are distinct characters. By Thoma \cite{t}, the indecomposable 
characters of $\Fin(\mathbb{N})$ are precisely the functions
\[
\chi_{(\beta; \gamma)}(g) = 
\prod_{k=2}^{\infty} 
(\,\sum_{i=1}^{\infty} \beta_{i}^{k} + 
(-1)^{k+1} \sum_{i=1}^{\infty} \gamma_{i}^{k}\,)
^{c_{k}(g)},
\] 
where $\beta = (\, \beta_{i} \,)_{i \in \mathbb{N}^{+}}$,
$\gamma = (\, \gamma_{i} \,)_{i \in \mathbb{N}^{+}}  
\in D[\,0,1\,]$ are such that $\sum_{i=1}^{\infty} \beta_{i} + \sum_{i=1}^{\infty} \gamma_{i} \leq 1$, and
$c_{k}(g)$ is the number of cycles of length $k$ in the cyclic decomposition of the permutation $g$.

Next suppose that $G$ is a countably infinite group and let $\Sub_{G}$ be the compact space
of subgroups $H \leqslant G$. Then a Borel probability measure $\nu$ on $\Sub_{G}$ which is invariant 
under the conjugation action of $G$ on $\Sub_{G}$ is called an {\em invariant random subgroup\/} or $IRS$.
For example, suppose that $G$ acts via measure-preserving maps on the 
Borel probability space $(\, Z, \mu \,)$ and let 
$f: Z \to \Sub_{G}$ be the
$G$-equivariant map defined by
\[
z \mapsto G_{z} = \{\, g \in G \mid g \cdot z = z \,\}.
\]
Then the corresponding {\em stabilizer distribution\/} $\nu = f_{*}\mu$ is an IRS 
of $G$. In fact, by a result of Ab\'{e}rt-Glasner-Virag \cite{agv}, every IRS of $G$ can be realized as the 
stabilizer distribution of a suitably chosen measure-preserving action. Moreover, 
by Creutz-Peterson \cite{cp}, if $\nu$ is an ergodic IRS of $G$, then $\nu$ is the stabilizer distribution 
of an ergodic action $G \curvearrowright (\, Z, \mu\,)$. 
If $\nu$ is an IRS of $G$, then we can define a corresponding
character $\chi_{\nu}$ by
\[
\chi_{\nu}(g) = \nu(\, \{\, H \in \Sub_{G} \mid g \in H \,\}\,).
\]
Equivalently, 
$\chi_{\nu}(g) = \mu(\, \Fix_{Z}(g)\,)$, where 
$G \curvearrowright (\, Z, \mu\,)$ is any measure-preserving action with stabilizer distribution $\nu$.  

In order to describe the ergodic IRSs of $\Fin(\mathbb{N})$, let $\alpha = (\, \alpha_{i}\,)_{i \in \mathbb{N}^{+}} 
\in D[\,0,1\,]$ be such that 
$\sum_{i=1}^{\infty} \alpha_{i} \leq 1$ and
let $\alpha_{0} = 1 - \sum_{i=1}^{\infty} \alpha_{i}$. 
Then we can define a probability measure $p_{\alpha}$ on $\mathbb{N}$ by $p_{\alpha}(\,\{\,i\,\}\,) = \alpha_{i}$. 
Let $\mu_{\alpha}$ be the corresponding product
probability measure on $\mathbb{N}^{\mathbb{N}}$. Then $\Fin(\mathbb{N})$ acts ergodically 
on $(\, \mathbb{N}^{\mathbb{N}}, \mu_{\alpha}\,)$ via the shift action
$(\, g \cdot \xi\,)(n) = \xi(\, g^{-1}(n)\,)$. 
For each $\xi \in \mathbb{N}^{\mathbb{N}}$ and $i \in \mathbb{N}$, let
$B^{\xi}_{i} = \{\, n \in \mathbb{N} \mid \xi(n) = i \,\}$. Then for
$\mu_{\alpha} \text{-a.e. } \xi \in \mathbb{N}^{\mathbb{N}}$, the following statements
are equivalent for all $i \in \mathbb{N}$.
\begin{enumerate}
\item[(a)] $\alpha_{i} > 0$.
\item[(b)] $B^{\xi}_{i} \neq \emptyset$.
\item[(c)] $B^{\xi}_{i}$ is infinite.
\item[(d)] $\lim_{n \to \infty} |B^{\xi}_{i} \cap \{\, 0,1, \cdots , n-1 \,\}|/n = \alpha_{i} > 0$.
\end{enumerate} 
In this case, we will say that $\xi$ is 
{\em $\mu_{\alpha}$-generic\/}.
First suppose that $\alpha_{0} \neq 1$, so that 
$I = \{\, i \in \mathbb{N}^{+} \mid \alpha_{i} > 0 \,\} \neq \emptyset$. Let 
$S_{\alpha} = \bigoplus_{i \in I} C_{i}$ be the restricted
direct product of the cyclic groups
$C_{i} = \{\, \pm 1 \,\}$ of order 2. ({\em Warning:\/} throughout this paper, we will regard $S_{\alpha}$ as a 
multiplicative group.)
Then for each subgroup $A \leqslant S_{\alpha}$, we can define a corresponding $\Fin(\mathbb{N})$-equivariant 
Borel map
\begin{align*}
f^{A}_{\alpha} : \mathbb{N}^{\mathbb{N}} &\to \Sub_{\Fin(\mathbb{N})} \\
                                     \xi &\mapsto H_{\xi}
\end{align*} 
as follows. If $\xi$ is $\mu_{\alpha}$-generic, then $H_{\xi} = s_{\xi}^{-1}(A)$,
where $s_{\xi}$ is the homomorphism
\begin{align*}
s_{\xi}: \bigoplus_{i \in I} \Fin( B^{\xi}_{i}) &\to \bigoplus_{i \in I} C_{i} \\
                                (\, \pi_{i} \,) &\mapsto (\, \sgn(\pi_{i}) \,).
\end{align*}
Otherwise, if $\xi$ is not $\mu_{\alpha}$-generic, then we let $H_{\xi} = 1$. Let
$\nu^{A}_{\alpha} = (f^{A}_{\alpha})_{*}\mu_{\alpha}$ be the corresponding ergodic
IRS of $\Fin(\mathbb{N})$.
Finally, if $\alpha_{0} = 1$, then we define $S_{\alpha} = 1$ and $\nu_{\alpha}^{E_{\alpha}} = \delta_{1}$.

\begin{Thm} \label{T:irs}
If $\nu$ is an ergodic IRS of $\Fin(\mathbb{N})$, then there
exists $\alpha$, $A$ as above such that 
$\nu = \nu^{A}_{\alpha}$. 
\end{Thm}

\begin{Remark} \label{R:irs}
There exist examples of sequences $\alpha$ and distinct subgroups $A$, $A^{\prime} \leqslant S_{\alpha}$
such that $\nu^{A}_{\alpha} = \nu^{A^{\prime}}_{\alpha}$.
For example, suppose that $\alpha_{1} = \alpha_{2} = 1/2$.
Then if $A_{1} = C_{1} \oplus 1$ and 
$A_{2} = 1 \oplus C_{2}$, then 
$\nu^{A_{1}}_{\alpha} = \nu^{A_{2}}_{\alpha}$.  
However, since for 
$\mu_{\alpha} \text{-a.e. } \xi \in \mathbb{N}^{\mathbb{N}}$, 
\[
\lim_{n \to \infty} |B^{\xi}_{i} \cap \{\, 0,1, \cdots , n-1 \,\}|/n = \alpha_{i}, 
\]
it follows that if
$\alpha \neq \alpha^{\prime}$ and $A$, $A^{\prime}$ are subgroups of $S_{\alpha}$,
$S_{\alpha^{\prime}}$, then $\nu^{A}_{\alpha} \neq \nu^{A^{\prime}}_{\alpha^{\prime}}$. 
\end{Remark}

\begin{Remark} \label{R:irs2}
Suppose that 
$\alpha = (\, \alpha_{i}\,)_{i \in \mathbb{N}^{+}}$ 
is such that there exist $i \in \mathbb{N}^{+}$ 
with $\alpha_{i} = \alpha_{i+1} > 0$. Let 
$A \leqslant S_{\alpha}$ be any subgroup and let
$\mathcal{S}^{A}_{\alpha}$ be the set of subgroups
$H \in \Sub_{\Fin(\mathbb{N}}$ such that there
exists a $\mu_{\alpha}$-generic 
$\xi \in \mathbb{N}^{\mathbb{N}}$ with
$\xi \overset{f^{A}_{\alpha}}{\mapsto} H_{\xi} = H$.
Then the map $\xi \mapsto H_{\xi} \in \mathcal{S}^{A}_{\alpha}$
is not injective; and it is easily seen that
there does not exist a $\Fin(\mathbb{N})$-equivariant
Borel map $H \mapsto \xi_{H}$ from $\mathcal{S}^{A}_{\alpha}$ 
to $\mathbb{N}^{\mathbb{N}}$ such that $H = H_{\xi_{H}}$. 
\end{Remark}

The relationship between the indecomposable characters
of $\Fin(\mathbb{N})$ and its ergodic IRSs is most
obvious when $A = S_{\alpha}$: in this case,
it is easily checked that 
\[
\chi_{\nu^{S_{\alpha}}_{\alpha}}(g) = 
\prod_{k=2}^{\infty} 
(\,\sum_{i=1}^{\infty} \alpha_{i}^{k} \,)^{c_{k}(g)}
= \chi_{(\alpha;\Bar{0})}(g),
\]
where $\Bar{0} \in D[\,0,1\,]$ denotes the
identically zero sequence. In particular, it follows
that $\chi_{\nu^{S_{\alpha}}_{\alpha}}$ is indecomposable.
Conversely, the somewhat ad hoc proof of 
Thomas-Tucker-Drob \cite[Theorem 9.2]{tt} shows
that if $A \neq S_{\alpha}$, then 
$\chi_{\nu^{A}_{\alpha}}$ is decomposable. (We
will give a more informative proof below.)
Of course, this result implies that if 
$\chi_{(\beta; \gamma)}$ is an indecomposable character
with $\gamma \neq \Bar{0}$, then there does not
exist an ergodic IRS $\nu$ of $\Fin(\mathbb{N})$
such that $\chi_{\nu} = \chi_{(\beta; \gamma)}$.  

In order to understand how arbitrary indecomposable
characters $\chi_{(\beta; \gamma)}$ are related to
the ergodic IRSs of $\Fin(\mathbb{N})$,
let $\widehat{S}_{\alpha}$ be the compact group of 
homomorphisms $\sigma: S_{\alpha} \to \{\, \pm 1\,\}$.
For each $i \in I$, let $c_{i}$ be the generator of
$C_{i} \leqslant S_{\alpha}$; and for each homomorphism
$\sigma \in \widehat{S}_{\alpha}$, let
$\sigma(i) = \sigma(c_{i})$. Then for each
$\sigma \in \widehat{S}_{\alpha}$, we can define
an indecomposable character of $\Fin(\mathbb{N})$ by
\addtocounter{equation}{3}
\begin{equation} \label{E:magic}
\chi^{\sigma}_{\alpha}(g) = 
\prod_{k=2}^{\infty} 
(\,\sum_{i \in I} \sigma(i)^{k+1} \alpha_{i}^{k} \,)^{c_{k}(g)}.
\end{equation}

\addtocounter{Thm}{1}
\begin{Remark} \label{R:thoma}
Suppose that $\sigma \in \widehat{S}_{\alpha}$. Let
$\beta = (\, \beta_{j} \,)_{j \in \mathbb{N}^{+}}
\in D[\,0,1\,]$ be the list 
(possibly augmented by a sequence of zeros) in decreasing
magnitude of the $\alpha_{i}$, $i \in I$, such that 
$\sigma(i) = 1$; and let 
$\gamma = (\, \gamma_{j} \,)_{j \in \mathbb{N}^{+}}
\in D[\,0,1\,]$ be the list 
(possibly augmented by a sequence of zeros) in decreasing
magnitude of the $\alpha_{i}$, $i \in I$, such that 
$\sigma(i) = -1$. Then clearly 
$\chi^{\sigma}_{\alpha} = \chi_{(\beta; \gamma)}$.

Conversely, if $\chi_{(\beta; \gamma)}$ is any indecomposable
character of $\Fin(\mathbb{N})$, then there exists
a sequence $\alpha = (\, \alpha_{i} \,)_{i \in \mathbb{N}^{+}}
\in D[\,0,1\,]$ and a homomorphism
$\sigma \in \widehat{S}_{\alpha}$ such that
$\chi_{(\beta; \gamma)} = \chi^{\sigma}_{\alpha}$.    
\end{Remark}

\begin{Ex} \label{E:char}
For later use, notice that if $\sigma \in \widehat{S}_{\alpha}$
is the trivial homomorphism such that $\sigma(s) = 1$ for all
$s \in S_{\alpha}$, then 
$\chi^{\sigma}_{\alpha} = \chi_{(\alpha;\Bar{0})}
= \chi_{\nu^{S_{\alpha}}_{\alpha}}$.
\end{Ex}

For each subgroup $A \leqslant S_{\alpha}$, let
$\widehat{(S_{\alpha}/A)}$ be
the compact subgroup of those $\sigma \in \widehat{S}_{\alpha}$
such that $\sigma(a) = 1$ for all $a \in A$ and
let $\mu^{A}_{\alpha}$ be the Haar probability
measure on $\widehat{(S_{\alpha}/A)}$. The following result
describes the relationship 
between the indecomposable characters of $\Fin(\mathbb{N})$ 
and its ergodic invariant random subgroups. 

\addtocounter{equation}{2}
\begin{Thm} \label{T:char}
If $\alpha$, $A$ are as above, then for each 
$g \in \Fin(\mathbb{N})$,
\begin{equation} \label{E:integral}
\chi_{\nu^{A}_{\alpha}}(g) = 
\int_{\sigma \in \widehat{(S_{\alpha}/A)}} 
\chi^{\sigma}_{\alpha}(g) \, d\mu^{A}_{\alpha}. 
\end{equation} 
\end{Thm}

\smallskip

\begin{Cor} \label{C:char}
$\chi_{\nu^{A}_{\alpha}}$ is indecomposable if and
only if $A = S_{\alpha}$.\\ 
\qed
\end{Cor}

The following corollary shows that for every indecomposable character $\chi$ of $\Fin(\mathbb{N})$, there exists a canonically associated ergodic IRS. 

\begin{Cor} \label{C:char2}
If $\chi$ is any indecomposable
character of $\Fin(\mathbb{N})$, then there
exists an $\alpha \in D[\,0,1\,]$
and a subgroup $A \leqslant S_{\alpha}$ such that
\[
\chi_{\nu^{A}_{\alpha}} =
\frac{1}{2} \chi_{\nu^{S_{\alpha}}_{\alpha}} +
\frac{1}{2} \chi. 
\] 
\end{Cor}

\begin{proof}
Applying Thoma's classification \cite{t}, let
$\beta = (\, \beta_{i} \,)_{i \in \mathbb{N}^{+}}$,
$\gamma = (\, \gamma_{i} \,)_{i \in \mathbb{N}^{+}}$
be such that $\chi = \chi_{(\beta; \gamma)}$. 
Let $\alpha = (\, \alpha_{i}\,)_{i \in \mathbb{N}^{+}}$
be a list in decreasing magnitude of the entries of the 
sequences $\beta = (\, \beta_{i}\,)_{i \in \mathbb{N}^{+}}$
and $\gamma = (\, \gamma_{i}\,)_{i \in \mathbb{N}^{+}}$.
Let $\sigma \in \widehat{S}_{\alpha}$ be such that
\[
\sigma(i) =
\begin{cases}
1,           &\text{if $\alpha_{i} = \beta_{j}$ for some
$j \in \mathbb{N}^{+}$;} \\
-1,           &\text{if $\alpha_{i} = \gamma_{j}$ for some
$j \in \mathbb{N}^{+}$;}
\end{cases}
\]
and let $A = \{\, s \in S_{\alpha} \mid \sigma(s) = 1\,\}$.
Then
\[
\chi_{\nu^{A}_{\alpha}} 
= \frac{1}{2} \chi_{\nu^{S_{\alpha}}_{\alpha}} +
\frac{1}{2} \chi^{\sigma}_{\alpha} 
= \frac{1}{2} \chi_{\nu^{S_{\alpha}}_{\alpha}} +
\frac{1}{2} \chi_{(\beta; \gamma)}
= \frac{1}{2} \chi_{\nu^{S_{\alpha}}_{\alpha}} +
\frac{1}{2} \chi.
\]
\end{proof}

The remainder of this paper is organized as follows. 
In Section \ref{S:char}, we will prove Theorem \ref{T:char} 
via an easy application of Fubini's Theorem. However,
this proof will give no insight into the meaning
of equations (\ref{E:magic}) and (\ref{E:integral}). 
In Section \ref{S:asymptotic}, via an application
of the Pointwise Ergodic Theorem, we will 
interpret the integral (\ref{E:integral}) as
an asymptotic limit of the Clifford decompositions
of a naturally associated sequence of permutation
characters of finite symmetric groups; and we
will interpret each Thoma character
$\chi_{(\beta; \gamma)} = \chi^{\sigma}_{\alpha}$
as an asymptotic limit of a naturally associated sequence 
of characters induced from linear characters of Young
subgroups of finite symmetric groups. Finally, in
Section \ref{S:irs}, slightly correcting the argument of Vershik \cite{v2}, we will present the proof of Theorem \ref{T:irs}.

Throughout this paper, we will identify $n$ with the
set $\{\, 0,1, \cdots, n-1\,\}$ and we will write
$S_{n} = \Sym(n)$.

\section*{Acknowledgments}
I would like to thank Richard Lyons for a very helpful discussion concerning the decomposition of permutation
characters. Thanks are also due to the Hausdorff Research Institute for Mathematics for its hospitality during the
preparation of this paper.

\section{The proof of Theorem \ref{T:char}} \label{S:char} 

In this section, we will present the proof of 
Theorem \ref{T:char}; i.e. that if $\alpha$, $A$ are
as in Section \ref{S:intro}, then for each 
$g \in \Fin(\mathbb{N})$,
\[
\chi_{\nu^{A}_{\alpha}}(g) = 
\int_{\sigma \in \widehat{(S_{\alpha}/A)}} 
\chi^{\sigma}_{\alpha}(g) \, d\mu^{A}_{\alpha}. 
\]
Clearly we can suppose that $g \neq 1$. Let
$g = g_{1} \cdots g_{t}$ be the  
decomposition of $g$ into a product of nontrivial
cycles; and for each $1 \leq \ell \leq t$, let
$g_{\ell}$ be an $k_{\ell}$-cycle. Notice that for each
$\sigma \in \widehat{(S_{\alpha}/A)}$,
\[
|\chi^{\sigma}_{\alpha}(g)| = 
|\prod_{\ell=1}^{t} 
(\,\sum_{i \in I} \sigma(i)^{k_{\ell}+1} 
\alpha_{i}^{k_{\ell}} \,)| \leq
\prod_{\ell=1}^{t} 
(\,\sum_{i \in I} \alpha_{i}^{k_{\ell}} \,)
= \chi_{(\alpha;\Bar{0})}(g); 
\] 
and it follows that
\[
\int_{\sigma \in \widehat{(S_{\alpha}/A)}} 
|\chi^{\sigma}_{\alpha}(g)| \, d\mu^{A}_{\alpha} \leq
\chi_{(\alpha;\Bar{0})}(g) < \infty.
\]
Hence, applying Fubini's theorem, we obtain that
\begin{align*}
\int_{\sigma \in \widehat{(S_{\alpha}/A)}} 
\chi^{\sigma}_{\alpha}(g) \, d\mu^{A}_{\alpha}
&= \int_{\sigma \in \widehat{(S_{\alpha}/A)}}
\prod_{\ell=1}^{t} 
(\,\sum_{i \in I} \sigma(i)^{k_{\ell}+1} 
\alpha_{i}^{k_{\ell}} \,) \, d\mu^{A}_{\alpha} \\
&= \sum_{\underline{i} \in I^{t}}
\int_{\sigma \in \widehat{(S_{\alpha}/A)}}
(\, \prod_{\ell=1}^{t} \sigma(i_{\ell})^{k_{\ell} +1}
\prod_{\ell=1}^{t} \alpha_{i_{\ell}}^{k_{\ell}}\,)\,
d\mu^{A}_{\alpha} \\
&= \sum_{\underline{i} \in I^{t}}
\alpha_{i_{1}}^{k_{1}} \cdots \alpha_{i_{t}}^{k_{t}} 
\int_{\sigma \in \widehat{(S_{\alpha}/A)}}
\prod_{\ell=1}^{t} \sigma(i_{\ell})^{k_{\ell} +1}
d\mu^{A}_{\alpha}, 
\end{align*}  
where each $\underline{i} = (\, i_{1}, \cdots, i_{t}\,)$. 

On the other hand, for each $i \in I$, let 
$\sgn_{i}: \Fin(\mathbb{N}) \to C_{i} \leqslant S_{\alpha}$
be the homomorphism such that for each $h \in \Fin(\mathbb{N})$
and $j \in I$, the $j^{\text{th}}$ component of $\sgn_{i}(h)$
is given by
\[
\sgn_{i}(h)_{j} = 
\begin{cases}
\sgn(h) &\text{if $j = i$;} \\
1       &\text{if $j \neq i$.} 
\end{cases}
\]
Then it is clear that
\[
\chi_{\nu^{A}_{\alpha}}(g) = \sum_{\underline{i} \in I^{t}}
\theta_{\underline{i}}(g) 
\alpha_{i_{1}}^{k_{1}} \cdots \alpha_{i_{t}}^{k_{t}}, 
\]
where
\[
\theta_{\underline{i}}(g) =
\begin{cases}
1, &\text{if $\sgn_{i_{1}}(g_{1}) \times \cdots 
\times \sgn_{i_{t}}(g_{t}) \in A$;} \\
0, &\text{if $\sgn_{i_{1}}(g_{1}) \times \cdots 
\times \sgn_{i_{t}}(g_{t}) \notin A$.}
\end{cases}
\]

Thus, in order to prove Theorem \ref{T:char},
it is enough to show that for all $\underline{i} \in I^{t}$,
\begin{equation} \label{E:point}
\theta_{\underline{i}}(g) =
\int_{\sigma \in \widehat{(S_{\alpha}/A)}}
\prod_{\ell=1}^{t} \sigma(i_{\ell})^{k_{\ell} +1}
d\mu^{A}_{\alpha}.
\end{equation}
Let $c_{i_{\ell}}$ be the generator of $
C_{i_{\ell}} = \{\, \pm 1 \,\}$. 
In the proof of (\ref{E:point}), we will make use of
the observation that if $\sigma \in \widehat{S}_{\alpha}$, 
then
\[
\sigma(\sgn_{i_{\ell}}(g_{\ell})) = \sigma(c_{i_{\ell}}^{k_{\ell} +1})
= \sigma(c_{i_{\ell}})^{k_{\ell} +1} = \sigma(i_{\ell})^{k_{\ell} +1};
\]
and hence we have that
\begin{align*}
\sigma(i_{1})^{k_{1}+1} \times \cdots \times \sigma(i_{t})^{k_{t}+1} 
&=\sigma(\sgn_{i_{1}}(g_{1})) \times \cdots 
\times \sigma(\sgn_{i_{t}}(g_{t})) \\ 
&= \sigma(\sgn_{i_{1}}(g_{1}) \times \cdots 
\times \sgn_{i_{t}}(g_{t})) 
\end{align*}
First suppose that 
$\sgn_{i_{1}}(g_{1}) \times \cdots \times \sgn_{i_{t}}(g_{t}) \in A$. Then for each $\sigma \in \widehat{(S_{\alpha}/A)}$,
\[
\sigma(i_{1})^{k_{1}+1} \times \cdots \times \sigma(i_{t})^{k_{t}+1} 
= \sigma(\sgn_{i_{1}}(g_{1}) \times \cdots 
\times \sgn_{i_{t}}(g_{t})) \\ 
= 1;
\]
and so $\int_{\sigma \in \widehat{(S_{\alpha}/A)}}
\prod_{\ell=1}^{t} \sigma(i_{\ell})^{k_{\ell} +1}
d\mu^{A}_{\alpha} = 1$. Next suppose that
\[
s = \sgn_{i_{1}}(g_{1}) \times \cdots \times \sgn_{i_{t}}(g_{t}) \notin A.
\]
For each $\varepsilon \in \{\, \pm 1 \,\}$, let
$V_{\varepsilon} = \{\, \sigma \in \widehat{(S_{\alpha}/A)}
\mid \sigma(s) = \varepsilon \,\}$.
Then clearly we have that
$\mu^{A}_{\alpha}(V_{1}) = \mu^{A}_{\alpha}(V_{-1}) = 1/2$;
and if $\sigma \in V_{\varepsilon}$, then
\[
\sigma(i_{1})^{k_{1}+1} \times \cdots \times \sigma(i_{t})^{k_{t}+1} 
= \sigma(s) = \varepsilon.
\]
Hence $\int_{\sigma \in \widehat{(S_{\alpha}/A)}}
\prod_{\ell=1}^{t} \sigma(i_{\ell})^{k_{\ell} +1}
d\mu^{A}_{\alpha} = 0$. This completes the proof of
Theorem \ref{T:char}. 

\section{An asymptotic interpretation of Theorem \ref{T:char}} 
\label{S:asymptotic}

In this section, via an application
of the Pointwise Ergodic Theorem, we will 
interpret the integral (\ref{E:integral}) as
an asymptotic limit of the Clifford decompositions
of a naturally associated sequence of permutation
characters of finite symmetric groups. 

Suppose that $G = \bigcup_{n \in \mathbb{N}} G_{n}$ is the 
union of the strictly increasing chain of finite subgroups $G_{n}$ and that $G \curvearrowright (\, Z, \mu\,)$ 
is an ergodic action on a Borel probability space. Then the following theorem is a special case of more general results of Vershik \cite{ver} and Lindenstrauss \cite{l}.

\begin{PET} \label{T:pet}
With the above hypotheses, if $B \subseteq Z$ is a $\mu$-measurable subset, then for $\mu$-a.e. $z \in Z$,
\[
\mu(B) = \lim_{n \to \infty} \frac{1}{|G_{n}|} \,|\{\, g \in G_{n} \mid g \cdot z \in B \,\}|.
\]
\end{PET}

In particular, the Pointwise Ergodic Theorem applies when $B$ is the $\mu$-measurable subset 
$\Fix_{Z}(g) = \{\, z \in Z \mid g \cdot z = z \,\}$ 
for some $g \in G$. For each $z \in Z$ and $n \in \mathbb{N}$, let $\Omega_{n}(z) = \{\, g \cdot z \mid g \in G_{n}\,\}$ be the corresponding $G_{n}$-orbit. Then, 
as pointed out in Thomas-Tucker-Drob \cite[Theorem 2.1]{tt2}, the following result is an easy
consequence of the Pointwise Ergodic Theorem.

\begin{Thm} \label{T:point}
With the above hypotheses, for $\mu$-a.e. $z \in Z$, 
for all $g \in G$,
\[
\mu(\, \Fix_{X}(g) \,) = \lim_{n \to \infty} |\, \Fix_{\Omega_{n}(z)}(g)\,|/|\, \Omega_{n}(z)\,|.  
\]
\end{Thm}

Of course, the permutation group
$G_{n} \curvearrowright \Omega_{n}(z)$ is isomorphic to 
$G_{n} \curvearrowright G_{n}/H_{n}$,
where $G_{n}/H_{n}$ is the set of cosets of 
$H_{n} = \{\, g \in G_{n} \mid g \cdot z = z \,\}$ in $G_{n}$;
and so
\addtocounter{equation}{1}
\begin{equation} \label{E:ind}
\Fix_{\Omega_{n}(z)}(g)\,|/|\, \Omega_{n}(z)\,|
= 1^{G_{n}}_{H_{n}}(g)/[G_{n}:H_{n}]. 
\end{equation}
Suppose that there exists a subgroup 
$H_{n} \leqslant K_{n} \leqslant G_{n}$
with $H_{n} \trianglelefteq K_{n}$. Then
$1^{G_{n}}_{H_{n}} = (\,1^{K_{n}}_{H_{n}}\,)^{G_{n}}$;
and, by Clifford's Theorem \cite[Theorem 11.5]{cr},
\[
1^{K_{n}}_{H_{n}} = \sum_{\theta \in \Irr_{H_{n}}(K_{n})}
\theta(1)\,\theta,
\]
where $\Irr_{H_{n}}(K_{n})$ is the set of irreducible
characters $\theta$ of $K_{n}$ such that
$\theta(h) = \theta(1)$ for all $h \in H_{n}$. (Thus
$\Irr_{H_{n}}(K_{n})$ can be naturally identified
with the set of irreducible characters of
the quotient group $K_{n}/H_{n}$.) It follows that
\begin{equation} \label{E:ind2}
1^{G_{n}}_{H_{n}} = (\,1^{K_{n}}_{H_{n}}\,)^{G_{n}} =
\sum_{\theta \in \Irr_{H_{n}}(K_{n})}
\theta(1)\,\theta^{G_{n}}.
\end{equation}

Now suppose that $\alpha \in D[\,0,1\,]$ and 
$A \leqslant S_{\alpha}$ are as in Section \ref{S:intro}
and let $\nu^{A}_{\alpha}$ be the corresponding
ergodic IRS of $\Fin(\mathbb{N})$. Applying 
Creutz-Peterson \cite{cp}, let $\nu^{A}_{\alpha}$ 
be the stabilizer distribution of the ergodic action
$\Fin(\mathbb{N}) \curvearrowright (\, Z, \mu\,)$
and let
\[
\chi_{\nu^{A}_{\alpha}}(g) = \mu(\, \Fix_{Z}(g) \,)
= \nu^{A}_{\alpha}(\, \{\, H \in \Sub_{\Fin(\mathbb{N})}
\mid g \in H \,\}\,)
\] 
be the corresponding character. Express
$\Fin(\mathbb{N}) = \bigcup_{n \in \mathbb{N}} S_{n}$,
where $S_{n} = \Sym(n)$; and for each
$z \in Z$ and $n \in \mathbb{N}$, let 
$\Omega_{n}(z) = \{\, g \cdot z \mid g \in S_{n}\,\}$ 
be the corresponding $S_{n}$-orbit. (As a matter of
convention, we set $\Sym(0) = 1$.) Then for 
$\mu$-a.e. $z \in Z$, we have that
\[
\chi_{\nu^{A}_{\alpha}}(g) = \mu(\, \Fix_{X}(g) \,) =
\lim_{n \to \infty} |\, \Fix_{\Omega_{n}(z)}(g)\,|/|\, \Omega_{n}(z)\,|.  
\]
Fix such an element $z \in Z$ and let 
$H = \{\, h \in \Fin(\mathbb(N) \mid h \cdot z = z \,\}$ be 
the corresponding point stabilizer. Then we can suppose that
there exists a $\mu_{\alpha}$-generic point
$\xi \in \mathbb{N}^{\mathbb{N}}$ such that 
$H = s_{\xi}^{-1}(A)$,
where $s_{\xi}$ is the homomorphism
\begin{align*}
s_{\xi}: \bigoplus_{i \in I} \Fin( B^{\xi}_{i}) 
&\to S_{\alpha} = \bigoplus_{i \in I} C_{i} \\
(\, \pi_{i} \,) &\mapsto (\, \sgn(\pi_{i}) \,).
\end{align*}
For each $n \in \mathbb{N}$, let $H_{n} = H \cap S_{n}$ and
let $I_{n} = \{\, i \in I \mid 
B^{\xi}_{i} \cap n \neq \emptyset\,\}$; and for 
each $i \in I_{n}$, let $B^{n}_{i} = B^{\xi}_{i} \cap n$. 
Let $K_{n} = \bigoplus_{i \in I_{n}} 
\Sym(B^{n}_{i})$ be the corresponding Young subgroup. Then 
$H_{n} \trianglelefteq K_{n} \leqslant S_{n}$ and
\[
K_{n}/H_{n} \cong \bigoplus_{i \in I_{n}}C_{i}/
(\,A \cap \bigoplus_{i \in I_{n}}C_{i}\,)
\]
is an elementary abelian 2-group; and so each
character $\sigma \in \Irr_{H_{n}}(K_{n})$ is linear.
Hence, applying equations (\ref{E:ind}) and (\ref{E:ind2}), 
we obtain that for each $g \in S_{n}$,
\begin{equation} \label{E:ind3}
\begin{split}
|\, \Fix_{\Omega_{n}(z)}(g)\,|/|\, \Omega_{n}(z)\,|
&= \frac{1}{[S_{n}:H_{n}]}
\sum_{\sigma \in \Irr_{H_{n}}(K_{n})} \sigma^{S_{n}}(g)\\
&= \sum_{\sigma \in \Irr_{H_{n}}(K_{n})}
\frac{1}{[K_{n}:H_{n}]}\,\, 
\frac{\sigma^{S_{n}}(g)}{\sigma^{S_{n}}(1)} \\
&= \sum_{\sigma \in \Irr_{H_{n}}(K_{n})}
\frac{1}{|\Irr_{H_{n}}(K_{n})|}\,\, 
\frac{\sigma^{S_{n}}(g)}{\sigma^{S_{n}}(1)}.
\end{split}
\end{equation}
Slightly abusing notation, for each
character $\sigma \in \Irr_{H_{n}}(K_{n})$, we will
also denote the corresponding homomorphism
$\bigoplus_{i \in I_{n}}C_{i} \to \{\, \pm 1 \,\}$
by $\sigma$. For each $i \in I_{n}$, let $c_{i}$ be 
the generator of $C_{i}$; and for each character
$\sigma \in \Irr_{H_{n}}(K_{n})$, let
$\sigma(i) = \sigma(c_{i})$. If
$g \in S_{n}$ is a $k$-cycle and
$\sigma \in \Irr_{H_{n}}(K_{n})$, then
\[
\frac{\sigma^{S_{n}}(g)}{\sigma^{S_{n}}(1)} = 
\frac{1}{|S_{n}|} \sum_{i \in I_{n}} \sigma(i)^{k+1}
|\{\, s \in S_{n} \mid sgs^{-1} \in \Sym(B^{n}_{i}) \,\}|
= \sum_{i \in I_{n}} \sigma(i)^{k+1}
\frac{\binom{|B^{n}_{i}|}{k}}{\binom{n}{k}}.
\]
If we fix $i \in I$ and let $n \to \infty$, then we
have that 
\[
\binom{|B^{n}_{i}|}{k}/\binom{n}{k} \approx
(|B^{n}_{i}|/n)^{k} \to \alpha_{i}^{k}
\quad \text{ as } n \to \infty.
\]
Hence we see that
\begin{align*}
\chi_{\nu^{A}_{\alpha}}(g) &=
\lim_{n \to \infty} |\, \Fix_{\Omega_{n}(z)}(g)\,|/|\, \Omega_{n}(z)\,| \\ 
&= \lim_{n \to \infty} 
\sum_{\sigma \in \Irr_{H_{n}}(K_{n})}
\frac{1}{|\Irr_{H_{n}}(K_{n})|}\,\, 
\frac{\sigma^{S_{n}}(g)}{\sigma^{S_{n}}(1)} \\
&= \int_{\sigma \in \widehat{(S_{\alpha}/A)}} 
(\,\sum_{i \in I} \sigma(i)^{k+1} \alpha_{i}^{k}\,) 
\, d\mu^{A}_{\alpha} \\
&= \int_{\sigma \in \widehat{(S_{\alpha}/A)}} 
\chi^{\sigma}_{\alpha}(g) \, d\mu^{A}_{\alpha}; 
\end{align*}
and, more generally, if $g \in \Fin(\mathbb{N})$
has cycle decomposition $g = g_{1} \cdots g_{t}$,
where each $g_{\ell}$ is a $k_{\ell}$-cycle, then
we see that
\begin{align*}
\chi_{\nu^{A}_{\alpha}}(g) &=
\lim_{n \to \infty} |\, \Fix_{\Omega_{n}(z)}(g)\,|/|\, \Omega_{n}(z)\,| \\ 
&= \lim_{n \to \infty} 
\sum_{\sigma \in \Irr_{H_{n}}(K_{n})}
\frac{1}{|\Irr_{H_{n}}(K_{n})|}\,\, 
\frac{\sigma^{S_{n}}(g)}{\sigma^{S_{n}}(1)} \\
&= \int_{\sigma \in \widehat{(S_{\alpha}/A)}} 
\prod_{1 \leq \ell \leq t} (\,\sum_{i \in I} 
\sigma(i)^{k_{\ell}+1} \alpha_{i}^{k_{\ell}}\,) 
\, d\mu^{A}_{\alpha} \\
&= \int_{\sigma \in \widehat{(S_{\alpha}/A)}} 
\chi^{\sigma}_{\alpha}(g) \, d\mu^{A}_{\alpha}. 
\end{align*}
Thus we can interpret Theorem \ref{T:char} as
an asymptotic limit of the Clifford decompositions
of a naturally associated sequence of permutation
characters of finite symmetric groups; and we
can interpret each Thoma character
$\chi_{(\beta; \gamma)} = \chi^{\sigma}_{\alpha}$
as an asymptotic limit of a naturally associated sequence 
of characters induced from linear characters of Young
subgroups of finite symmetric groups.

\section{The ergodic invariant random subgroups of 
$\Fin(\mathbb{N})$} \label{S:irs}

In this section, slightly correcting the argument
of Vershik \cite{v2}, we will present the proof of
Theorem \ref{T:irs}. The classification of the ergodic
IRSs of $\Fin(\mathbb{N})$ will be based upon the following
two insights of Vershik.
\begin{enumerate}
\item[(i)] If $H \leqslant \Fin(\mathbb{N})$ is a random subgroup, then the corresponding $H$-orbit equivalence relation is a random equivalence relation on $\mathbb{N}$; and these have been classified by Kingman \cite{k}.
\item[(ii)] The induced action of $H$ on each infinite orbit $\Omega \subseteq \mathbb{N}$ can be determined 
via an application of 
Wielandt's Theorem \cite[Satz 9.4]{w}, which states that if $\Omega$ is an infinite set, then $\Alt(\Omega)$ and $\Fin(\Omega)$ are the only primitive subgroups of $\Fin(\Omega)$. 
\end{enumerate}

The proof of Theorem \ref{T:irs} will also make
use of an elementary result concerning imprimitive actions 
of finitary permutation groups. Recall that a 
transitive subgroup $H \leqslant \Sym(\Omega)$ is said to 
act {\em imprimitively\/} if there exists a nontrivial
proper $H$-invariant equivalence relation $E$ on 
$\Omega$. In this case, following the usual practice,
we will refer to the $E$-classes as {\em $E$-blocks\/}.
Now suppose that $\Omega$ is an infinite set and 
$H \leqslant \Fin(\Omega)$ is an imprimitive subgroup. 
Of course, in this case, if $E$ is a nontrivial proper 
$H$-invariant equivalence relation on $\Omega$, then 
the $E$-blocks must 
be finite. The following result, which is a 
variant of Neumann \cite[Lemma 2.2]{n}, implies 
that we can always make a ``canonical'' choice of a 
nontrivial proper $H$-invariant 
equivalence relation.\footnote{In \cite{v2}, Vershik
suggests choosing the minimal nontrivial $H$-invariant equivalence relation. However, there exist examples
of imprimitive finitary groups $H$ with more than one
minimal nontrivial $H$-invariant equivalence relation.
On the other hand, see Claim \ref{C:min}.}   

\begin{Lem} \label{L:choice}
If $\Omega$ is an infinite set and 
$H \leqslant \Fin(\Omega)$ is an imprimitive subgroup,
then at least one of the following two conditions holds:
\begin{enumerate}
\item[(i)] there exists a {\em unique\/} maximal proper
$H$-invariant equivalence relation $E^{\Omega}_{\mx}$ on 
$\Omega$; 
\item[(ii)] there exists a {\em unique\/} proper
$H$-invariant equivalence relation $E^{\Omega}_{\mn}$ on 
$\Omega$ which is minimal subject to the condition that 
$E^{\Omega}_{\mn}$ contains every minimal nontrivial $H$-invariant equivalence relation on $\Omega$.
\end{enumerate}
\end{Lem}

\begin{Def} \label{D:can}
If $\Omega$ is an infinite set and 
$H \leqslant \Fin(\Omega)$ is an imprimitive subgroup,
then the associated {\em canonical equivalence relation\/}
$E^{\Omega}_{\can}$ is defined to be $E^{\Omega}_{\mx}$ if condition \ref{L:choice}(i) holds, and is defined to be 
$E^{\Omega}_{\mn}$ otherwise.
\end{Def}

The proof of Lemma \ref{L:choice} will make use of the
following observation.

\begin{Claim} \label{C:min}
If $\Omega$ is an infinite set and 
$H \leqslant \Fin(\Omega)$ is an imprimitive subgroup,
then there exist only finitely many minimal nontrivial 
$H$-invariant equivalence relations on $\Omega$.
\end{Claim}

\begin{proof}[Proof of Claim \ref{C:min}]
Suppose that $\{\, E_{n} \mid n \in \mathbb{N} \,\}$ are 
distinct minimal nontrivial $H$-invariant equivalence relations
on $\Omega$. Fix some element $\omega_{0} \in \Omega$; and
for each $n \in \mathbb{N}$, let $\Delta_{n}$ be the 
$E_{n}$-block such that $\omega_{0} \in \Delta_{n}$. Notice that
if $n \neq m$, then $\Delta_{n} \cap \Delta_{m}$ is
a block; and hence by the minimality of $E_{n}$, $E_{m}$,
we must have that 
$\Delta_{n} \cap \Delta_{m} = \{\, \omega_{0} \,\}$.
For each $n \in \mathbb{N}$, choose an element
$d_{n} \in \Delta \smallsetminus \{\, \omega_{0} \,\}$.
Let $\pi \in H$ satisfy $\pi(\omega_{0}) = d_{0}$. 
If $n > 0$, then $d_{0} \in \pi(\Delta_{n}) \neq \Delta_{n}$ 
and so $\pi(d_{n}) \neq d_{n}$, which contradicts
the fact that $\pi \in \Fin(\mathbb{N})$.
\end{proof}

\begin{proof}[Proof of Lemma \ref{L:choice}]
Once again, fix some element $\omega_{0} \in \Omega$.
First suppose that there exists a maximal nontrivial proper
$H$-invariant equivalence relation $E$ on $\Omega$.
Then the set $\Omega/E$ of $E$-classes is infinite and $H$ acts as a primitive group of finitary permutations on
$\Omega/E$. Applying Wielandt's Theorem, it
follows that $H$ induces at least $\Alt(\Omega/E)$
on $\Omega/E$. Suppose that $E^{\prime} \neq E$ is a 
second maximal proper $H$-invariant equivalence 
relation on $\Omega$. Let $\Delta$ be the $E$-block
such that $\omega_{0} \in \Delta$ and let $\Delta^{\prime}$ 
be the $E^{\prime}$-block such that 
$\omega_{0} \in \Delta^{\prime}$. Then there exists 
$d \in \Delta^{\prime} \smallsetminus \Delta$. Since $H$ 
acts 2-transitively on $\Omega/E$, it follows that
the orbit $H_{\{\Delta\}} \cdot d$ is infinite;
and since 
$[\, H_{\{\Delta\}}: H_{(\Delta)}\,] < \infty$,
it follows that the orbit $H_{(\Delta)} \cdot d$ is 
also infinite. But this means that $\Delta^{\prime}$
is infinite, which is a contradiction.

Next suppose that there does not exist a maximal proper
$H$-invariant equivalence relation on $\Omega$. Then there
exists a strictly increasing sequence
\[ 
E_{0} \subset E_{1} \subset \cdots \subset E_{n} \subset \cdots
\]
of proper $H$-invariant equivalence relations
on $\Omega$. Let $E = \bigcup_{n \in \mathbb{N}}E_{n}$.
Then $E$ is an $H$-invariant equivalence relation such that
every $E$-class is infinite and it follows that 
$E = \Omega^{2}$. For each $n \in \mathbb{N}$, let
$\Delta_{n}$ be the $E_{n}$-block such that 
$\omega_{0} \in \Delta_{n}$. Then clearly 
$\bigcup_{n \in \mathbb{N}} \Delta_{n} = \Omega$.

Applying Claim \ref{C:min}, let $R_{1}, \cdots, R_{m}$
be the finitely many minimal nontrivial 
$H$-invariant equivalence relations on $\Omega$;
and for each $1 \leq \ell \leq m$, let $\Phi_{\ell}$
be the $R_{\ell}$-block such that $\omega_{0} \in \Phi_{\ell}$.
Then there exists $n \in \mathbb{N}$ such that
$\Phi_{\ell} \subseteq \Delta_{n}$ for all
$1 \leq \ell \leq m$; and it follows that 
$R_{\ell} \subseteq E_{n}$ for all $1 \leq \ell \leq m$.
Finally, letting $E^{H}_{\mn}$ be the intersection of all
the proper $H$-invariant equivalence relations $E$
such that $R_{\ell} \subseteq E$ for all $1 \leq \ell \leq m$,
it is clear that $E^{H}_{\mn}$ satisfies condition
\ref{L:choice}(ii).  
\end{proof}

Finally, before beginning the proof of Theorem \ref{T:irs},
we will present a brief discussion of 
Kingman's Theorem \cite{k}. Let 
\[
\ER_{\mathbb{N}} = 
\{\, E \in 2^{\mathbb{N} \times \mathbb{N}} \mid
E \text{ is an equivalence relation on } \mathbb{N} \,\}.
\]
Then $\ER_{\mathbb{N}}$ is a compact space and
$\Fin(\mathbb{N}) \curvearrowright \ER_{\mathbb{N}}$
via the shift action
\[ 
(\, g \cdot E\,)(n,m) = E(g^{-1}(n),g^{-1}(m)).
\]
As expected, a $\Fin(\mathbb{N})$-invariant Borel
probability measure $m$ on $\ER_{\mathbb{N}}$ is
called an {\em invariant random equivalence relation\/}.
For example, let 
$\alpha = (\, \alpha_{i}\,)_{i \in \mathbb{N}^{+}} 
\in D[\,0,1\,]$ be such that 
$\sum_{i=1}^{\infty} \alpha_{i} \leq 1$ and
let $\alpha_{0} = 1 - \sum_{i=1}^{\infty} \alpha_{i}$. 
Let $\mu_{\alpha}$ be the corresponding product
measure on $\mathbb{N}^{\mathbb{N}}$, as defined
in Section \ref{S:intro}; and for each $\xi \in \mathbb{N}^{\mathbb{N}}$ and $i \in \mathbb{N}$, let
$B^{\xi}_{i} = \{\, n \in \mathbb{N} \mid \xi(n) = i \,\}$. 
Then we can define a $\Fin(\mathbb{N})$-equivariant
map $\xi \overset{\varphi_{\alpha}}{\mapsto} E_{\xi}$
from $\mathbb{N}^{\mathbb{N}}$ to $\ER_{\mathbb{N}}$ 
by letting $E_{\xi}$ correspond to the partition
\[
\mathbb{N} = \bigsqcup_{n \in B^{\,\xi}_{0}}\{\,n\,\} \sqcup \bigsqcup_{i > 0} B^{\,\xi}_{i};
\]  
and it follows that 
$m_{\alpha} = (\varphi_{\alpha})_{*}\mu_{\alpha}$
is an ergodic random invariant equivalence relation.
The following theorem is due to Kingman \cite{k}.

\begin{Thm} \label{T:kingman}
If $m$ is an ergodic random invariant equivalence relation,
then there exists $\alpha$ as above such that
$m = m_{\alpha}$.
\end{Thm}

\begin{Remark} \label{R:eqn}
If $\alpha = (\, \alpha_{i}\,)_{i \in \mathbb{N}^{+}}$ 
is such that there exist $i \in \mathbb{N}^{+}$ 
with $\alpha_{i} = \alpha_{i+1} > 0$, then the map
$\xi \mapsto E_{\xi}$ is not injective, and 
there does not exist a $\Fin(\mathbb{N})$-equivariant
Borel map $E \mapsto \xi_{E}$ from $\ER_{\mathbb{N}}$ 
to $\mathbb{N}^{\mathbb{N}}$ such that $E = E_{\xi_{E}}$. 
\end{Remark}

The proof of Theorem \ref{T:irs} will also make use
of the following easy observation.

\begin{Lem} \label{L:finite}
If $m$ is an ergodic random invariant equivalence relation,
then $m$ concentrates the equivalence relations
$E \in \ER_{\mathbb{N}}$ such that every $E$-class is
either infinite or a singleton.
\end{Lem}

\begin{proof}
While this result is an immediate consequence of 
Theorem \ref{T:kingman}, it seems worthwhile
to give an elementary proof. So
suppose that $m$ is a counterexample. Then, by ergodicity,
there exists a fixed integer $k > 1$ such that
\[
m(\, \{\, E \in \ER_{\mathbb{N}} \mid \text{ There exists
an $E$-class of size } k \,\}) = 1.
\]
For each $S \in [\, \mathbb{N}\,]^{k}$, let $C_{S}$
be the event that $S$ is an $E$-class. Since
$\Fin(\mathbb{N})$ acts transitively on
$[\, \mathbb{N}\,]^{k}$, there exists 
a fixed real $r > 0$ such that $m(C_{S}) = r$ 
for all $S \in [\, \mathbb{N}\,]^{k}$.
But, since the events 
$\{\, C_{S} \mid 0 \in S \in [\, \mathbb{N}\,]^{k}\,\}$
are mutually exclusive, this is impossible.
\end{proof}

We are now ready to begin the proof of Theorem \ref{T:irs}.
So suppose that $\nu$ is an ergodic IRS of $\Fin(\mathbb{N})$.
Clearly we can suppose that $\nu \neq \delta_{1}$.

\begin{Lem} \label{L:alt}
For $\nu$-a.e.\ $H \in \Sub_{\Fin(\mathbb{N})}$,
if $\Omega \subseteq \mathbb{N}$ is a nontrivial
$H$-orbit, then $\Omega$ is infinite and $H$
induces at least $\Alt(\Omega)$ on $\Omega$.
\end{Lem}

\begin{proof}
Suppose not. Recall that, by Wielandt's Theorem \cite{w}, if $\Omega$ is an infinite set, then $\Alt(\Omega)$ and $\Fin(\Omega)$ are the only primitive subgroups of $\Fin(\Omega)$. It follows that for 
$\nu$-a.e.\ $H \in \Sub_{\Fin(\mathbb{N})}$, either
there exists a nontrivial finite $H$-orbit, or else
there exists an infinite $H$-orbit on which $H$ acts
imprimitively. For each such $H \in \Sub_{\Fin(\mathbb{N})}$, 
let $R_{H}$ be the equivalence relation on $\mathbb{N}$
such that if $n \neq m \in \mathbb{N}$, then 
$n \mathbin{R_{H}} m$ if and only if $n$, $m$ lie in the same
$H$-orbit $\Omega$ and either:
\begin{enumerate}
\item[(i)] $\Omega$ is a nontrivial finite $H$-orbit; or
\item[(ii)] $\Omega$ is an infinite imprimitive
$H$-orbit and $n \mathbin{E^{\Omega}_{\can}} m$.
\end{enumerate}
Otherwise, let $R_{H}$ be the trivial equivalence
relation on $\mathbb{N}$. Then clearly the map
$H \overset{\psi}{\mapsto} R_{H}$ is
$\Fin(\mathbb{N})$-equivariant, and 
hence $m = \psi_{*}\nu$ is an ergodic
invariant random equivalence relation.
But $m$ concentrates on the 
$E \in \ER_{\mathbb{N}}$ with a nontrivial
finite $E$-class, which contradicts Lemma \ref{L:finite}.
\end{proof}

The proof of the following lemma will make use of the
notion of a diagonal subgroup, which is defined
as follows. Suppose that $r \geq 2$ and that
$\Omega_{1}, \cdots , \Omega_{r}$ are countably infinite
sets. Then
$D \leqslant \bigoplus_{1 \leq \ell \leq r} \Alt(\Omega_{\ell})$ 
is said to be a {\em diagonal subgroup\/}
if there exist isomorphisms 
$\pi_{\ell}: \Alt(\Omega_{1}) \to \Alt(\Omega_{\ell})$ for 
$2 \leq \ell \leq r$ such that 
\[
D = \{\, (g, \pi_{2}(g), \cdots, \pi_{r}(g)) 
\mid g \in \Alt(\Omega_{1}) \,\}.
\]
Recall that every automorphism
of $\Alt(\mathbb{N})$ is the restriction of an
inner automorphism of the group 
$\Sym(\mathbb{N})$ of all permutations
of $\mathbb{N}$. It follows that for each
$2 \leq \ell \leq r$, there exists a unique
bijection $T_{\ell}: \Omega_{1} \to \Omega_{\ell}$
such that $\pi_{\ell}(g) = T_{\ell}gT_{\ell}^{-1}$.
Let $T_{1}$ be the identity map on $\Omega_{1}$;
and for each $1 \leq k,\ell \leq r$, let
$T_{k,\ell} = T^{-1}_{\ell}T_{k}$. Then we will write
$D = \Diag(\, \bigoplus_{1 \leq \ell \leq r} 
\Alt(\Omega_{\ell})\,)$,
and say that $D$ is the diagonal subgroup determined
by the bijections 
$\{\, T_{k,\ell} \mid 1 \leq k,\ell \leq r \,\}$.
Finally, in the degenerate case when $r = 1$, we will take 
$\Alt(\Omega_{1})$ to be the only diagonal subgroup of 
$\Alt(\Omega_{1})$ and we will take $T_{1,1}$ to be the
identity map on $\Omega_{1}$.

\begin{Lem} \label{L:alt2}
For $\nu$-a.e.\ $H \in \Sub_{\Fin(\mathbb{N})}$,
if $\{\, \Omega_{i} \mid i \in I \,\}$ is the set
of nontrivial $H$-orbits, then each $\Omega_{i}$
is infinite and 
$\bigoplus_{i \in I}\Alt(\Omega_{i}) \leqslant H$.
\end{Lem}

\begin{proof}
By Lemma \ref{L:alt}, for $\nu$-a.e.\ 
$H \in \Sub_{\Fin(\mathbb{N})}$,
if $\Omega \subseteq \mathbb{N}$ is a nontrivial
$H$-orbit, then $\Omega$ is infinite and $H$
induces at least $\Alt(\Omega)$ on $\Omega$.
Let $H$ be such a subgroup, let 
$\{\, \Omega_{i} \mid i \in I \,\}$ is the set
of nontrivial $H$-orbits, and let
$K = H \cap \bigoplus_{i \in I}\Alt(\Omega_{i})$.

\begin{Claim} \label{C:product}
There exists a partition 
$\{\, F_{j} \mid j \in J \,\}$ of
$I$ into finite subsets such that
\[
K = \bigoplus_{j \in J} 
\Diag( \bigoplus_{k \in F_{j}} \Alt(\Omega_{k})), 
\]
where the diagonal subgroups are determined by
unique bijections 
$T_{k,\ell}: \Omega_{k} \to \Omega_{\ell}$ for 
$k$, $\ell \in F_{j}$.  
\end{Claim}

\begin{proof}[Sketch proof of Claim \ref{C:product}]
For each $g \in K$, let $g = \prod_{i \in I}g_{i}$, where 
$g_{i} \in \Alt(\Omega_{i})$, and let 
$s(g) = \{\, i \mid g_{i} \neq 1 \,\}$. Then clearly each $s(g)$ is a finite subset of $I$. Let 
$\mathcal{P} = \{\, F_{j} \mid j \in J \,\}$ be the collection of minimal subsets $A \subseteq I$ such that there exists 
$1 \neq g \in K$ with $s(g) = A$. 
Then, using the simplicity
of the infinite alternating group and the fact
that $K$ projects onto each $\Alt(\Omega_{i})$, it is easily checked that $\mathcal{P}$ is a partition of $I$ and that
\[
K = \bigoplus_{j \in J} 
\Diag( \bigoplus_{k \in F_{j}} \Alt(\Omega_{k}))
\] 
for some collection of bijections 
\[ 
\{\, T_{k,\ell}: \Omega_{k} \to \Omega_{\ell}
\mid k, \ell \in F_{j} \text{ for some } j \in J \,\}. 
\]
\end{proof}
Let $R_{H}$ be the equivalence relation on $\mathbb{N}$
such that if $n \neq m \in \mathbb{N}$, then 
$n \mathbin{R_{H}} m$ if and only if there exists $j \in J$
and $k \neq \ell \in F_{j}$ such that
$T_{k,\ell}(n) = m$. Clearly the map
$H \overset{\psi}{\mapsto} R_{H}$ is
$\Fin(\mathbb{N})$-equivariant, and hence $m = \psi_{*}\nu$ is an ergodic invariant random equivalence relation. 
Applying Lemma \ref{L:finite}, since every $R_{H}$-class
is finite, it follows that each $|F_{j}| = 1$.
\end{proof}

Let $H \overset{p}{\mapsto} E_{H}$ be the 
$\Fin(\mathbb{N})$-equivariant map from 
$\Sub_{\Fin(\mathbb{N}}$ to $\ER_{\mathbb{N}}$ 
such that $E_{H}$ is the $H$-orbit equivalence relation.
Then $p_{*}\nu$ is an ergodic invariant random
equivalence relation; and hence, applying 
Theorem \ref{T:kingman}, it follows that
there exists an
$\alpha = (\, \alpha_{i}\,)_{i \in \mathbb{N}^{+}} 
\in D[\,0,1\,]$ such that $p_{*}\nu = m_{\alpha}$.
Since $\nu \neq \delta_{1}$, it follows that 
$\alpha_{0} \neq 1$. Let 
$I = \{\, i \in \mathbb{N}^{+} \mid \alpha_{i} > 0 \,\}$.
Then for $\nu$-a.e.\ $H \in \Sub_{\Fin(\mathbb{N})}$,
there exists a $\mu_{\alpha}$-generic 
$\xi_{H} \in \mathbb{N}^{\mathbb{N}}$ such that
the $H$-orbit decomposition is given by
\addtocounter{equation}{9}
\begin{equation} \label{E:orbit}
\mathbb{N} = \bigsqcup_{n \in B^{\,\xi_{H}}_{0}}\{\,n\,\} \sqcup \bigsqcup_{i \in I} B^{\,\xi_{H}}_{i}.
\end{equation}
As we mentioned in Remark \ref{R:eqn}, if 
$\alpha = (\, \alpha_{i}\,)_{i \in \mathbb{N}^{+}}$ 
is such that there exist $i \in \mathbb{N}^{+}$ 
with $\alpha_{i} = \alpha_{i+1} > 0$, then the map
$\xi \mapsto E_{\xi}$ is not injective. In more detail, let
$\equiv$ be the equivalence relation on $I$ defined by
\[
k \equiv \ell \quad \Longleftrightarrow \quad 
\alpha_{k} = \alpha_{\ell};
\]
and let $I = \bigsqcup_{j \in J}I_{j}$ be the decomposition
of $I$ into $\equiv$-classes. Then clearly each $I_{j}$ is
finite. Let $P = \prod_{j \in J}\Sym(I_{j})$ be the full
direct product of the finite groups $\Sym(I_{j})$, and
let $P \curvearrowright 
(\, \mathbb{N}^{\mathbb{N}}, \mu_{\alpha} \,)$ 
be the measure-preserving action defined by
\[
(\, (\pi_{j})_{j \in J} \cdot \xi \,)(n) =
\begin{cases}
\pi_{j}(\xi(n)),           &\text{if $\xi(n) \in I_{j}$;} \\
\xi(n),           &\text{if $\xi(n) \in \mathbb{N} \smallsetminus I$.}
\end{cases}
\]
Then $P$ is a (possibly trivial) compact group; and if 
$\xi$, $\xi^{\prime} \in \mathbb{N}^{\mathbb{N}}$
are $\mu_{\alpha}$-generic, then 
$E_{\xi} = E_{\xi^{\prime}}$ if and only if there exists
$(\pi_{j})_{j \in J} \in P$ such that
$(\pi_{j})_{j \in J} \cdot \xi = \xi^{\prime}$.

\addtocounter{Thm}{1}
\begin{Def} \label{D:generic}
A subgroup $H \in \Sub_{\Fin(\mathbb{N})}$ is said to be
{\em $\nu$-generic\/} if:
\begin{enumerate} 
\item[(i)] there exists a $\mu_{\alpha}$-generic 
$\xi_{H} \in \mathbb{N}^{\mathbb{N}}$ such that the $H$-orbit
decomposition is given by (\ref{E:orbit}); and
\item[(ii)] $H$ satisfies the conclusion of Lemma \ref{L:alt2}. \end{enumerate}
\end{Def}

Let $H \in \Sub_{\Fin(\mathbb{N})}$ be $\nu$-generic and let
$\xi_{H} \in \mathbb{N}^{\mathbb{N}}$ be the  
$\mu_{\alpha}$-generic function chosen so that
if $\alpha_{i}, \alpha_{i+1}, \cdots, \alpha_{i+s}$
is a nontrivial $\equiv$-class, then the corresponding
orbits $B^{\,\xi_{H}}_{i}, B^{\,\xi_{H}}_{i+1}, \cdots,
B^{\,\xi_{H}}_{i+s}$ are listed in the order of their
least elements. Let
\begin{align*}
s_{H}: \bigoplus_{i \in I} \Fin( B^{\xi_{H}}_{i}) &\to S_{\alpha} = \bigoplus_{i \in I} C_{i} \\
                                (\, \pi_{i} \,) &\mapsto (\, \sgn(\pi_{i}) \,)
\end{align*}
and let $A_{H} = s_{H}(H) \leqslant S_{\alpha}$.
Once again, let $P = \prod_{j \in J}\Sym(I_{j})$. Then the
natural action
$P \curvearrowright S_{\alpha} = \bigoplus_{i \in I} C_{i}$
induces a corresponding action 
$P \curvearrowright \Sub_{S_{\alpha}}$. For each 
$A \leqslant S_{\alpha}$, let $[\,A\,]$ be the
corresponding $P$-orbit. Since $P$ is a compact group,
it follows that
$\Sub_{S_{\alpha}}/P = \{\, [\,A\,] \mid 
A \in \Sub_{S_{\alpha}}\,\}$ 
is a standard Borel space. Furthermore, the Borel map
$H \mapsto [\,A_{H}\,]$ is clearly 
$\Fin(\mathbb{N})$-invariant. Hence, by ergodicity, there
exists a fixed $A  \in \Sub_{S_{\alpha}}$
such that $[\,A_{H}\,] = [\, A \,]$ for $\nu$-a.e.\
$H \in \Sub_{\Fin(\mathbb{N})}$. Let 
$X^{A}_{\alpha} \subseteq \Sub_{\Fin(\mathbb{N})}$
be the set of $\nu$-generic $H$ such that 
$[\,A_{H}\,] = [\, A \,]$. Then both $\nu$ and
$\nu^{A}_{\alpha}$ concentrate on $X^{A}_{\alpha}$.
Hence, in order to complete the proof of 
Theorem \ref{T:irs}, it is enough to show that
the action $\Fin(\mathbb{N}) \curvearrowright X^{A}_{\alpha}$
is uniquely ergodic. As we will explain, this is a 
straightforward consequence of the Pointwise Ergodic Theorem.

For each pair $F_{0}$, $F_{1}$ of finite disjoint subsets
of $\Fin(\mathbb{N})$, let
\[
U_{F_{0},F_{1}} = \{\, H \in \Sub_{\Fin(\mathbb{N})} \mid
F_{0} \subseteq H \text{ and } F_{1} \cap H = \emptyset \,\}.
\]
Then the sets $U_{F_{0},F_{1}}$ form a clopen basis of the
space $\Sub_{\Fin(\mathbb{N})}$; and thus it is enough to
show that $\nu(\,U_{F_{0},F_{1}}\,) = 
\nu^{A}_{\alpha}(\,U_{F_{0},F_{1}}\,)$ for all such
$F_{0}$, $F_{1}$. Hence, by the Pointwise Ergodic Theorem,
it is enough to show that if 
$H$, $H^{\prime} \in X^{A}_{\alpha}$, then 
\[
\lim_{n \to \infty} \frac{1}{|S_{n}|}
|\{\, g \in S_{n} \mid gHg^{-1} \in U_{F_{0},F_{1}} \,\}|
= \lim_{n \to \infty} \frac{1}{|S_{n}|}
|\{\, g \in S_{n} \mid 
gH^{\prime}g^{-1} \in U_{F_{0},F_{1}} \,\}|.
\]
Equivalently, letting $H_{n} = H \cap S_{n}$ and 
$H^{\prime}_{n} = H^{\prime} \cap S_{n}$, it is 
enough to show that
\addtocounter{equation}{1}
\begin{equation} \label{E:pet}
\begin{split}
\lim_{n \to \infty}& \frac{1}{|S_{n}|}
|\{\, g \in S_{n} \mid g^{-1}F_{0}g \subseteq H_{n}
\text{ and } g^{-1}F_{1}g \cap H_{n} = \emptyset \,\}| \\
&= \lim_{n \to \infty} \frac{1}{|S_{n}|}
|\{\, g \in S_{n} \mid g^{-1}F_{0}g \subseteq H^{\prime}_{n}
\text{ and } g^{-1}F_{1}g \cap H^{\prime}_{n} = \emptyset \,\}|. 
\end{split}
\end{equation}
To see this, first note that after changing our choice
of $\xi_{H^{\prime}}$ if necessary, we can suppose that
$A_{H} = A_{H^{\prime}}$. Next fix some $\varepsilon > 0$
and choose an integer $k \in I$ such that
$1 - \sum_{i=0}^{k} \alpha_{i} \ll \varepsilon$. 
Let $F_{0} \sqcup F_{1} \subseteq \Sym(d)$ and let
$n \gg d$ be such that 
$|\, |B_{i}^{\xi_{H}} \cap n|/n - \alpha_{i}|
\ll \varepsilon$ and 
$|\, |B_{i}^{\xi_{H^{\prime}}} \cap n|/n - \alpha_{i}|
\ll \varepsilon$ for all $0 \leq i \leq k$. Then
it is easily checked that if $n$ is sufficiently 
large, then
\[
\begin{split}
\biggl|\,\frac{1}{|S_{n}|}&
|\{\, g \in S_{n} \mid g^{-1}F_{0}g \subseteq H_{n}
\text{ and } g^{-1}F_{1}g \cap H_{n} = \emptyset \,\}| \\
&- \frac{1}{|S_{n}|}
|\{\, g \in S_{n} \mid g^{-1}F_{0}g \subseteq H^{\prime}_{n}
\text{ and } g^{-1}F_{1}g \cap H^{\prime}_{n} = \emptyset \,\}|\,\biggr| < \varepsilon;
\end{split} 
\]
and so (\ref{E:pet}) holds. This completes the proof of
Theorem \ref{T:irs}.

\end{document}